\documentclass[11pt,a4paper]{amsart}
\usepackage{amsmath}
\usepackage{hyperref}
\usepackage[utf8]{inputenc}
\usepackage[english]{babel}
\usepackage[T1]{fontenc}
\usepackage{url}
\usepackage{mathrsfs}
\usepackage{ifthen}
\usepackage{xparse}

\usepackage{tikz}
\usetikzlibrary{knots}
\usetikzlibrary{decorations.pathreplacing}
\usepackage{xcolor}
\usepackage{graphicx}

\usepackage{amssymb}
\usepackage{lscape}
\usepackage{multirow}
\usepackage{wrapfig} 

\usepackage{caption}
\usepackage{enumitem}

\newtheorem{thm}{Theorem}
\newtheorem{lem}{Lemma}
\newtheorem{corollary}{Corollary}
\newtheorem{conjecture}{Conjecture}

\theoremstyle{definition}

\theoremstyle{remark}
\newtheorem{remark}{Remark}

\newcommand{\numknownsol}{38 }

\newcommand{\Z}{\mathbb{Z}}

\newcounter{twistcounter}
\newcounter{listlength}   

\newcommand{\twist}[4]{
	\ifnum #1>0
	\foreach \i in {1,...,#1} {
		\draw[white, double=black, ultra thick, double distance = 1] 
		(#3+#2*#3, {#4*(\i-1)/#1}) to[out=90, in=-90, looseness=0.5] (#2*#3, {#4*(\i)/#1});
		\draw[white, double=black, ultra thick, double distance = 1] 
		(#2*#3, {#4*(\i-1)/#1}) to[out=90, in=-90, looseness=0.5] (#3+#2*#3, {#4*(\i)/#1});
	}
	\else
	\foreach \i in {1,...,-#1} {
		\draw[white, double=black, ultra thick, double distance = 1] 
		(#2*#3, {-#4*(\i-1)/#1}) to[out=90, in=-90, looseness=0.5] (#3+#2*#3, -{#4*(\i)/#1});
		\draw[white, double=black, ultra thick, double distance = 1] 
		(#3+#2*#3, {-#4*(\i-1)/#1}) to[out=90, in=-90, looseness=0.5] (#2*#3, -{#4*(\i)/#1});
	}
	\fi
}

\newcommand{\pretzelknot}[2]{
	\begin{tikzpicture}
		\setcounter{twistcounter}{0}
		\setcounter{listlength}{0} 
		\foreach \x in {#1} { \addtocounter{listlength}{1} }
		\foreach \x [count=\nn] in {#1} {
			\pgfmathparse{mod({2*\nn}, {2*\thelistlength})} 
			\let\nextnn\pgfmathresult
			\pgfmathparse{ifthenelse(abs(2*\nn+1 - \nextnn) > 1.1, 0.75, 2)}
			\let\loosenessvalue\pgfmathresult
			\pgfmathsetmacro{\scaledX}{4/(2*\the\numexpr\value{listlength}\relax-1)} 
			\pgfmathsetmacro{\xshift}{2*\nn - 2}
			\pgfmathsetmacro{\ysize}{2.5}
			
			\twist{\x}{\xshift}{\scaledX}{\ysize}
			\draw[line width = 1] 
			({(2*\nn-1)*\scaledX}, \ysize) 
			to[out=90, in=90, looseness=\loosenessvalue] 
			({\nextnn*\scaledX}, \ysize);
			
			\draw[line width = 1] 
			({(2*\nn-1)*\scaledX}, 0) 
			to[out=-90, in=-90, looseness=\loosenessvalue] 
			({\nextnn*\scaledX}, 0);
			\addtocounter{twistcounter}{1}
		}
		\node[anchor=center] at (2,-1.3) {#2};
	\end{tikzpicture}
}

\newcommand{\pretzellink}[3]{
	\begin{tikzpicture} 
		\setcounter{twistcounter}{0}
		\setcounter{listlength}{0} 
		\foreach \x in {#1} { \addtocounter{listlength}{1} }
		\foreach \x [count=\nn] in {#1} {
			\pgfmathparse{mod({2*\nn}, {2*\thelistlength})} 
			\let\nextnn\pgfmathresult
			\pgfmathparse{ifthenelse(abs(2*\nn+1 - \nextnn) > 1.1, 0.75, 2)}
			\let\loosenessvalue\pgfmathresult
			\pgfmathsetmacro{\scaledX}{4/(2*\the\numexpr\value{listlength}\relax-1)} 
			\pgfmathsetmacro{\xshift}{2*\nn - 2}
			\pgfmathsetmacro{\ysize}{2.5}
			
			\pgfmathtruncatemacro{\arrowdir}{ifthenelse((#3 < 0), 0, (ifthenelse((mod(\nn,2) == 1), 2, 1)))}
			\ifnum\arrowdir=1
			\def\toporientation{->} 
			\def\botorientation{<-}
			\else \ifnum\arrowdir=2
			\def\toporientation{<-} 
			\def\botorientation{->}
			\else
			\def\toporientation{->} 
			\def\botorientation{->}
			\fi
			\fi
			
			\twist{\x}{\xshift}{\scaledX}{\ysize}
			\draw[line width = 1, \toporientation] 
			({(2*\nn-1)*\scaledX}, \ysize) 
			to[out=90, in=90, looseness=\loosenessvalue] 
			({\nextnn*\scaledX}, \ysize);
			
			\draw[line width = 1, \botorientation] 
			({(2*\nn-1)*\scaledX}, 0) 
			to[out=-90, in=-90, looseness=\loosenessvalue] 
			({\nextnn*\scaledX}, 0);
			\addtocounter{twistcounter}{1}
		}
		\node[anchor=center] at (2,-1.3) {#2};
	\end{tikzpicture}
}

\title{Explicit Formulas for the Alexander Polynomial of Pretzel Knots}
\author{Yury Belousov}
\address{Yury Belousov\\
Leonhard Euler International Mathematical Institute in Saint Petersburg}
\email{bus99@yandex.ru}
\thanks{The work is supported by the Ministry of Science and Higher Education of the Russian Federation (agreement no. 075-15-2025-343).}

\begin{document}
	\begin{abstract}
		We provide explicit formulas for the Alexander polynomial of pretzel knots and establish several immediate corollaries, including the characterization of pretzel knots with a trivial Alexander polynomial. As an application, we construct a new family of knots that are topologically slice but not smoothly slice.
	\end{abstract}
	\maketitle

    \section*{Introduction}
	A \emph{pretzel link} \( P(q_1, q_2, \dots, q_n) \) is a link obtained by connecting \( n \) twisted bands, where each band has \( q_i \) half-twists. The sign of \( q_i \) determines the direction of the \(i\)-th twist region. A pretzel link is a knot if and only if either precisely one \( q_i \) is even, or all \( q_i \) are odd and \( n \) is odd (see Fig.~\ref{fig: pretzel knots} for examples). 
	
	Pretzel links were introduced by Reidemeister in~\cite{Reidemeister1932} and have been thoroughly studied since then. The Alexander polynomial~\cite{Alexander1928} of these knots has also been studied before --- see, for example, \cite{Nakagawa1987}, \cite{Hironaka2001}, \cite{Lecuona2015}, \cite{MironovMorozovSleptsov2015}, \cite{BaeLee2020} --- and explicit formulas for some special cases were found. Kim and Lee~\cite{KimLee2007} also found explicit formulas for the Conway polynomial for several classes of pretzel links with at least one even parameter.\footnote{The case where all parameters are odd is also discussed in the paper, although the authors do not provide an explicit formula.} 
	
	Nevertheless, general formulas for the Alexander polynomial of pretzel knots have not appeared in the literature. In this note, we fill this gap (see Theorem~\ref{thm: main formula}).

	\begin{figure}[h]
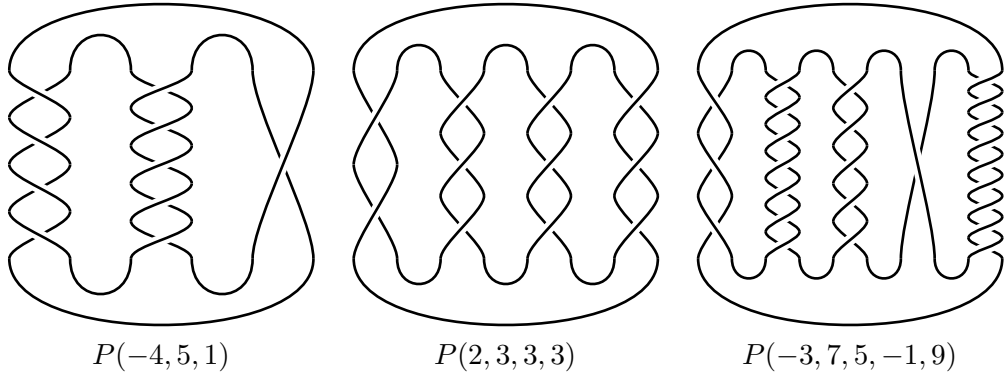

		\centering 
		\pretzelknot{-4,5, 1}{\(P(-4,5,1)\)}
		\pretzelknot{2,3,3,3}{\(P(2,3,3,3)\)}
		\pretzelknot{-3,7,5, -1, 9}{\(P(-3,7,5, -1, 9)\)}
		\caption{Examples of pretzel knots}
		\label{fig: pretzel knots}
	\end{figure}

    \subsection*{Conventions and notation} 
	\begin{itemize}
		\item If \(K = P(q_1,q_2, \dots, q_n)\) is a pretzel knot with exactly one even \( q_i \), we may assume, without loss of generality, that \(q_1\) is even. 
		\item We denote the Alexander polynomial of a link \(L\) by \(\Delta_L(t)\).
		\item We write \(f(t)\doteq g(t)\) to indicate that the Laurent polynomials \(f(t)\) and \(g(t)\) are equal up to multiplication by \(\pm t^k\) for some \(k \in \Z\).
		\item We denote by \(\sigma_k(q_1, q_2, \dots, q_n)\) the elementary symmetric polynomial of degree \( k \), that is, the sum of all products of \( k \) distinct variables among \( q_1, q_2, \dots, q_n \). For \( k < 0 \) we set \(\sigma_k(q_1, q_2, \dots, q_n) = 0\).
	\end{itemize}

	\begin{thm} \label{thm: main formula}
		Let \(K=P(q_1, q_2, \dots, q_n)\) be a pretzel knot. Then the Alexander polynomial \(\Delta_K(t)\) is given by the following cases:
		\begin{enumerate}[label=(\arabic*), ref=(\arabic*)]
			\item \label{case 1 thm main} If \( n \) is odd, \( q_1 \) is even, and \( q_2, \dots, q_n \) are odd,
			\[
			\Delta_K(t) \doteq \left(\prod_{i=2}^n \frac{1+t^{q_i}}{1+t} \right) 
			\left(1 + \frac{q_1}{2} (t^{-1} - t) \left(\sum_{i=2}^n\frac{t^{q_i}}{1+t^{q_i}}  - \frac{n-1}{2}\right) \right).
			\]
			\item \label{case 2 thm main} If \( n \) is even, \( q_1 \) is even, and \( q_2, \dots, q_n \) are odd,
			\[
			\Delta_K(t) \doteq \left(\prod_{i=2}^n \frac{1+t^{q_i}}{1+t} \right) 
			\left(t^{q_1/2} + \left(t^{q_1/2} - t^{-q_1/2}\right) \left(\sum_{i=2}^n\frac{t^{q_i}}{1+t^{q_i}}  - \frac{n}{2}\right) \right).
			\]        
			\item \label{case 3 thm main} If both \( n \) and all \( q_i \) are odd,
			\[
			\Delta_K(t) \doteq \frac{1}{2^{n-1}}\sum_{k=0}^{(n-1)/2} (t-1)^{2k}(t+1)^{n-1-2k} \sigma_{2k}(q_1, q_2, \dots, q_n).
			\]             
		\end{enumerate}
	\end{thm}
	
\subsection*{Structure of the paper}
    In Section~\ref{sec: main proof}, we provide an elementary proof of the explicit formulas, relying on the skein relation. 
    In Section~\ref{sec: corollaries}, we derive some straightforward corollaries, including the characterization of pretzel knots with a trivial Alexander polynomial and a formula for the knot determinant. 
\subsection*{Acknowledgments}
    The author would like to thank Vadim Stepaniuk for many helpful discussions.

\section{Proof of Theorem~\ref{thm: main formula}} \label{sec: main proof}
	Recall that the \emph{symmetrized Alexander polynomial} of a link \( L \) is a well-defined normalization of the Alexander polynomial\footnote{This symmetric normalization is due to Conway~\cite{Conway1970}. Accordingly, the invariant is also referred to as the \emph{Alexander--Conway polynomial}.}, given as a Laurent polynomial in \( t^{1/2} \) and characterized by the following two properties:
	(1) \(\Delta_L(t) = \Delta_L(t^{-1})\), and  
	(2) if \( L \) is a knot, then \(\Delta_L(1) = 1\).  
	The symmetrized Alexander polynomial satisfies the skein relation:
	\[
	\Delta_{L_+}(t) - \Delta_{L_-}(t) = (t^{1/2} - t^{-1/2})\Delta_{L_0}(t),
	\]
	where links \(L_+\), \(L_-\), and \(L_0\) differ by the local operation shown in Fig.~\ref{fig: skein relation}.
	
	\begin{figure}[h]
		\centering
		\begin{tikzpicture}[scale=0.9]
			\begin{scope}[shift={(-3,0)}, xscale = 0.5, yscale = 0.75]
				\draw[line width = 2, ->] (1,-1)  to (-1, 1);
				\draw[line width = 2, white, double=black] (-1,-1)  to (1, 1);
				\draw[line width = 2, ->] (-1,-1)  to (1, 1);
				\node at (0.0,-1.3) {\(L_+\)};
			\end{scope}
			
			\begin{scope}[shift={(0,0)}, xscale = 0.5, yscale = 0.75]
				\draw[line width = 2, ->] (-1,-1)  to (1, 1);
				\draw[line width = 2, white, double=black] (1,-1)  to (-1, 1);
				\draw[line width = 2, ->] (1,-1)  to (-1, 1);
				
				\node at (0.0,-1.3) {\(L_-\)};
			\end{scope}
			
			\begin{scope}[shift={(3,0)}, xscale = 0.5, yscale = 0.75]
				\draw[line width = 2, ->] (1,-1)  to [looseness = 1, out = 120, in = -120] (1, 1);
				\draw[line width = 2, ->] (-1,-1)  to [looseness = 1, out = 60, in = -60] (-1, 1);
				\node at (0.0,-1.3) {\(L_0\)};
			\end{scope}
		\end{tikzpicture}
		\caption{Illustration of the skein relation}    
		\label{fig: skein relation}    
	\end{figure}
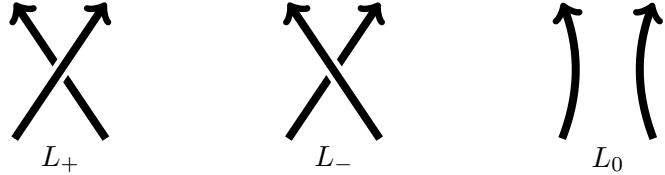
	
	To apply the skein relation, we also need explicit formulas for a special case of two-component pretzel links. The following lemma will be proved simultaneously with Theorem~\ref{thm: main formula}.  
	
	\begin{lem} \label{lem: two component pretzel}
		Let \(L=P(q_1, q_2, \dots, q_n)\) be an oriented two-component pretzel link, where \( n \) is even and all \( q_i \) are odd. Then the Alexander polynomial \(\Delta_L(t)\) is given by the following cases:
		\begin{enumerate}[label=(\arabic*), ref=(\arabic*)]
			\item \label{case 1 lemma} If the strands in each twist of \(L\) are co-directed (as in Fig.~\ref{fig: pretzel links}(A)), then
			\[
			\Delta_L(t) \doteq \frac{1}{(1+t)^{n-1}}\prod_{i=1}^n\left( 1+t^{q_i}\right) 
			\left(\sum_{i=1}^n\frac{t^{q_i}}{1+t^{q_i}} - \frac{n}{2}\right).
			\]
			\item \label{case 2 lemma} If the strands in each twist of \(L\) are oppositely directed (as in Fig.~\ref{fig: pretzel links}(B)), then
			\[
			\Delta_L(t) \doteq  \frac{1}{2^{n-1}}\sum_{k=1}^{n/2} (t-1)^{2k-1}(t+1)^{n-2k} \sigma_{2k-1}(q_1, q_2, \dots, q_n).
			\]
		\end{enumerate}
	\end{lem}
	\begin{figure}[h]
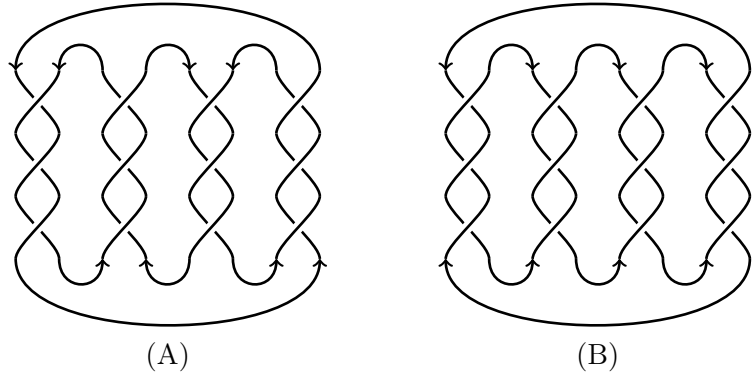

		\centering
		\pretzellink{3,3,3,3}{(A)}{1}
		\hspace{1cm}
		\pretzellink{3,3,3,3}{(B)}{-1}    
		\caption{Example of two-component pretzel link \(P(3,3,3,3)\) with different orientations}
		\label{fig: pretzel links}
	\end{figure}
	
	\subsection{Case~\ref{case 3 thm main} of Theorem~\ref{thm: main formula} and Case~\ref{case 2 lemma} of Lemma~\ref{lem: two component pretzel}}
	Let \(L = P(q_1, q_2,\dots, q_n)\) be a pretzel link (or a knot) where all \(q_i\) are odd. We will use the following notation:
	\begin{align} \label{eq: alex poly odd case}
		f_o(t;q_1,\dots, q_n) &= \frac{t^{-\frac{n-1}{2}}}{2^{n-1}}\sum_{k=0}^{(n-1)/2} (t-1)^{2k}(t+1)^{n-1-2k}\sigma_{2k}(q_1, \dots, q_n), \\
		f_e(t;q_1, \dots, q_n) &=\frac{t^{-\frac{n-1}{2}}}{2^{n-1}}\sum_{k=1}^{n/2} (t-1)^{2k-1}(t+1)^{n-2k}\sigma_{2k-1}(q_1, \dots, q_n).
	\end{align}	
	First, observe that
	\begin{align*}
		&f_o(t; 1,-1,q_1, \dots, q_n) = \frac{t^{-\frac{n+1}{2}}}{2^{n+1}}\sum_{k=0}^{(n+1)/2} (t-1)^{2k}(t+1)^{n+1-2k}\underbrace{\sigma_{2k}(1,-1,q_1, \dots, q_n)}_{= \sigma_{2k}(q_1,\dots, q_n) - \sigma_{2k-2}(q_1,\dots, q_n)} = \\
		&=\frac{t^{-\frac{n+1}{2}}}{2^{n+1}}\sum_{k=0}^{(n-1)/2} \underbrace{\left((t-1)^{2k}(t+1)^{n+1-2k} - (t-1)^{2k+2}(t+1)^{n-1-2k} \right)}_{4t(t-1)^{2k}(t+1)^{n-1-2k}}\sigma_{2k}(q_1, \dots, q_n) =\\
		&=f_o(t; q_1, q_2, \dots, q_n).
	\end{align*}
	A similar argument holds for \(f_e(t;q_1,\dots, q_n)\). 
	
	If all \(|q_i| = 1\) and \(m = |\sum_{i=1}^n q_i|\), then \(L\) is the \((2, m)\)-torus link (or knot), and  
	\begin{align*}        
		f_o(t; q_1, \dots, q_n)  &=  \frac{t^{-\frac{m-1}{2}}}{2^{m-1}}\sum_{k=0}^{(m-1)/2} (t-1)^{2k}(t+1)^{m-1-2k}\underbrace{\sigma_{2k}(1, \dots, 1)}_{=\binom{m}{2k}} = \\&= t^{-\frac{m-1}{2}}\sum_{k=0}^{m-1} (-t)^k = t^{-\frac{m-1}{2}}\frac{t^m+1}{t+1}.    
	\end{align*}
	By an analogous calculation
	\begin{align*}        
		f_e(t; q_1, \dots, q_n) = t^{-\frac{m-1}{2}}\frac{t^m-1}{t+1}.   
	\end{align*}   
	These are well-known formulas, confirming our induction base. 
	
	Now, assume that the induction hypothesis holds for \({K_1 = P(q_1, q_2, \dots, q_n)}\) and for \({K_0 = P(q_2, \dots, q_n)}\). We can apply the skein relation to compute \(\Delta_{K}(t)\) for \({K = P(q_1\pm2, q_2, \dots, q_n)}\). We consider the case \(q_1+2\); the argument for \(q_1-2\) is analogous. Applying the induction hypothesis, we obtain the following for odd \(n\):
	\begin{align*}
		\Delta_{P(q_1+2, q_2, \dots, q_n)}(t) = f_o(t; q_1, q_2, \dots,q_n) + (t^{1/2} - t^{-1/2})f_e(t; q_2, \dots, q_n).
	\end{align*}
	Similarly, for even \(n\) we have: 
	\begin{align*}
		\Delta_{P(q_1+2, q_2, \dots, q_n)}(t) = f_e(t; q_1, q_2, \dots,q_n) + (t^{1/2} - t^{-1/2})f_o(t; q_2, \dots, q_n).
	\end{align*}
	Since 
	\[
	\sigma_k(q_1+2, q_2, \dots, q_n) = \sigma_k(q_1,q_2,\dots,q_n) + 2\sigma_{k-1}(q_2, \dots, q_n), 
	\]
	this confirms that \( f_o \) and \( f_e \) satisfy the desired recursion, completing the proof.

	\subsection{Cases~\ref{case 1 thm main} and \ref{case 2 thm main} of Theorem~\ref{thm: main formula} and Case~\ref{case 1 lemma} of Lemma~\ref{lem: two component pretzel}}
	
	The proofs of these cases share a similar approach, so we present only the key steps, omitting direct computations. 
	
	The base case of the induction for Case~\ref{case 1 lemma} of Lemma~\ref{lem: two component pretzel} corresponds to the link \( P(\pm1, \pm1, \dots, \pm1) \), which is, once again, a two-component \( (2, m) \)-torus link, though with a different orientation. The base case of the induction for Cases~\ref{case 1 thm main} and \ref{case 2 thm main} of Theorem~\ref{thm: main formula} is \( P(0, q_2, \dots, q_n) \), which is a connected sum of \((2, q_i)\)-torus knots for each \( q_i \), regardless of the parity of \( n \). Recall that the Alexander polynomial of a connected sum is the product of the Alexander polynomials of its summands.
	
	Finally, we use the skein relation to confirm the validity of the given formulas.

\begin{remark}
    The explicit formulas were originally conjectured based on symbolic experiments. We first noticed that for a two-component link with odd parameters and co-directed strands (Case~\ref{case 1 lemma} in Lemma~\ref{lem: two component pretzel}), the Alexander polynomial has the form
    \[
        \frac{1}{(1+t)^{n-1}}\sum_{S\subseteq\{1,\dots,n\}}\left(\frac{n}{2}-|S|\right)t^{\left(\sum\limits_{i\in S} q_i\right)}.
    \]
    We then rewrote this expression into the form presented in Lemma~\ref{lem: two component pretzel} and applied the skein relation to derive the formulas for Cases~\ref{case 1 thm main} and \ref{case 2 thm main} in Theorem~\ref{thm: main formula}. For Case~\ref{case 3 thm main} in Theorem~\ref{thm: main formula}, we observed that the coefficients in the Alexander polynomial consistently corresponded to elementary symmetric polynomials. These observations guided the formulation of the theorem.\footnote{We also validated our conjectures by comparison with known data from KnotInfo~\cite{knotinfo}.}
\end{remark}

	\section{Corollaries}\label{sec: corollaries}
	\subsection{Determinant of pretzel knots}	
	The determinant of a pretzel knot follows directly from Theorem~\ref{thm: main formula}. This result reproduces the formula obtained by Kolay in \cite{Kolay2023}.
	
	\begin{corollary} \label{cor: det}
		The determinant of the pretzel knot \( K = P(q_1, q_2, \dots, q_n) \) is given by:
		\[
		\det\left(K\right) = |\sigma_{n-1}(q_1, q_2, \dots, q_n)| = \left|q_1q_2\dots q_n\left(\frac{1}{q_1}+\frac{1}{q_2}+\dots+\frac{1}{q_n}\right)\right|.
		\]
	\end{corollary}
	
	\begin{proof}
		Recall that \(\det(K) = |\Delta_K(-1)|\). Thus, to calculate the determinant of a pretzel knot, we need to consider the three cases in Theorem~\ref{thm: main formula}.  
		
		In Case~\ref{case 3 thm main}, the result follows by direct substitution. Cases~\ref{case 1 thm main} and \ref{case 2 thm main} involve routine calculus computations. We perform the computation for Case~\ref{case 1 thm main}; the argument for Case~\ref{case 2 thm main} is completely analogous.  
		
		The formula for Case~\ref{case 1 thm main} in Theorem~\ref{thm: main formula} consists of two factors. For the first one, we have  
		\[
		\lim_{t\to -1}\prod_{i=2}^n \frac{1+t^{q_i}}{1+t} = \prod_{i=2}^n q_i.
		\]
		For the second factor, using the expansion at \( t = -1 \), we obtain the following limits:
		\begin{align*}    
			&\lim_{t\to -1}\left(1 + \frac{q_1}{2}\left(t^{-1} - t\right)\left(\sum_{i=2}^n\frac{t^{q_i}}{1+t^{q_i}}  - \frac{n-1}{2}\right) \right) 
			= \\
			&=1+\frac{q_1}{2}\lim_{\varepsilon \to 0}\left(\underbrace{\left(\frac{1}{-1+\varepsilon} +1 - \varepsilon \right)}_{=-2\varepsilon + O(\varepsilon^2)}\left(\sum_{i=2}^n\underbrace{\left(\frac{(-1+\varepsilon)^{q_i}}{1+(-1+\varepsilon)^{q_i}}  - \frac{1}{2}\right)}_{=\frac{1}{2} - \frac{1}{q_i \varepsilon} + O(\varepsilon)} \right) \right) =\\
			&= 1 + q_1\left(\sum_{i=2}^n\frac{1}{q_i} \right) = q_1\left( \sum_{i=1}^n\frac{1}{q_i}\right).
		\end{align*}
		Combining these results, we obtain \( \sigma_{n-1}(q_1,q_2, \dots, q_n) \), as stated.
	\end{proof}

	\subsection{Degree of the Alexander polynomial of pretzel knots}	
	For pretzel knots with at least one even parameter, the genus was completely determined in~\cite{KimLee2007} using the formulas for the Conway polynomial. The case where all parameters are odd is more subtle (except for alternating knots; see Corollary~\ref{cor: genus}). In this case, the degree of the Alexander polynomial can take any even value from \( 0 \) to \( n-1 \), and its dependence on the parameters is not straightforward. 
	
	In this section, we examine two cases: when the Alexander polynomial of such knots has degree \( n-1 \) (the maximal possible) and when it is trivial. Before analyzing the general case, we first examine the special case of alternating pretzel knots, where the degree of the Alexander polynomial directly determines the genus.

	\begin{corollary} \label{cor: genus}
		Let \( K = P(q_1, q_2, \dots, q_n) \) be a pretzel knot with all \( q_i \) odd and of the same sign. Then its genus satisfies \( g(K) = \frac{n-1}{2} \).
	\end{corollary}
	
	\begin{proof}
		If all \(q_i\) have the same sign:
		\begin{enumerate}
			\item The knot is alternating, so \( g(K) = \frac{\deg \Delta_K(t)}{2} \). This follows from a classical result of Murasugi~\cite{Murasugi1958} and Crowell~\cite{Crowell1959}.
			\item Since \( \sigma_{2k}(q_1, q_2, \dots, q_n) \) is positive for all \( 0 \leq k \leq \frac{n-1}{2} \), it follows that \( \deg \Delta_K(t) = n-1 \). \qedhere
		\end{enumerate}
	\end{proof}

	\begin{corollary}
		\label{cor: trivial polynomial}
		Let \(K = P(q_1, q_2, \dots, q_n)\) be a pretzel knot where all \(q_i\) are odd. Then the following hold:
		\begin{enumerate}[label=(\arabic*), ref=(\arabic*)]
			\item The Alexander polynomial satisfies \( \deg \Delta_K(t) = n - 1 \) if and only if 
			\[
			\sum_{k=0}^{(n-1)/2} \sigma_{2k}(q_1, q_2, \dots, q_n) \neq 0.
			\]
			\item \label{case trivial Alex} The Alexander polynomial is trivial, i.e., \( \Delta_K(t) \doteq 1 \), if and only if
            \begin{equation}\label{eq: system trivial alex}
                \sigma_{2k}(q_1, q_2, \dots, q_n) = (-1)^{k} \binom{\frac{n-1}{2}}{k}, \qquad  0 \leq k \leq \frac{n-1}{2}.
            \end{equation}
		\end{enumerate}
	\end{corollary}

	\begin{proof}
		To simplify the notation, let us denote \( a_k = \sigma_k(q_1, q_2, \dots, q_n) \).
		
		Consider the Alexander polynomial in the form
		\begin{equation*}
			\Delta_K(t) = \frac{1}{2^{n-1}} \sum_{k=0}^{(n-1)/2} (t-1)^{2k}(t+1)^{n-1-2k} a_{2k}.
		\end{equation*}
		
		The polynomials \((t-1)^{2k}(t+1)^{n-1-2k}\) are palindromic; consequently, the coefficient of \( a_{2i} t^j \) in the expansion equals the coefficient of \( a_{2i} t^{n-1-j} \). Thus, \( \deg \Delta_K(t) = n - 1 \) if the coefficient of \( t^0 \) is nonzero. 
		Since \[\Delta_K(0) = \frac{1}{2^{n-1}} \sum_{k=0}^{(n-1)/2} a_{2k}, \] this completes the proof of the first case.
		
		The condition \( \Delta_K(t) \doteq 1 \) holds if and only if the only nonzero coefficient of \(\Delta_K(t)\) appears at \( t^{(n-1)/2} \) and equals 1. This requirement can be expressed as a system of linear equations:
		\begin{equation}
			\label{eq: system of linear equations}
			A \left(a_{0}, a_{2}, \ldots, a_{\frac{n-1}{2}}\right)^\top = (0, 0, \ldots, 0, 1)^\top,
		\end{equation}
		where \( A \in \mathbb{R}^{\frac{n+1}{2} \times \frac{n+1}{2}} \) is a matrix whose element \( A_{i, j} \) is equal to the coefficient of \( a_{2i - 2}t^{j-1} \) in \( \Delta_K(t) \).  
		
		The matrix \( A \) is non-degenerate. Indeed, if we consider \( u = t-1 \), then
		\[
		(t-1)^{2k}(t+1)^{n-1-2k} = u^{2k}(u+2)^{n-1-2k},
		\]
		and such polynomials are clearly linearly independent. Thus, there exists a unique solution to the system~\eqref{eq: system of linear equations}.  		
		Setting
		\(
		a_{2k} = (-1)^{k} \binom{\frac{n-1}{2}}{k},
		\)
		we compute:
		\begin{align*}
			\frac{1}{2^{n-1}} &\sum_{k=0}^{(n-1)/2} (-1)^{k} \binom{\frac{n-1}{2}}{k} (t-1)^{2k}(t+1)^{n-1-2k}=\\ 
			&= \frac{1}{2^{n-1}} \left((t+1)^2 - (t-1)^2 \right)^{\frac{n-1}{2}} 
			= t^{\frac{n-1}{2}}.\qedhere
		\end{align*}
    \end{proof}

    \subsection{Topologically slice pretzel knots}
    The Alexander polynomial of a non-trivial pretzel knot with an even parameter is never trivial. Hence, Case~\ref{case trivial Alex} of Corollary~\ref{cor: trivial polynomial} provides a complete characterization of all pretzel knots with a trivial Alexander polynomial. According to Freedman's theorem~\cite{Freedman1982}, such knots are topologically slice. Note that a complete classification of topologically slice odd pretzel knots with three twist blocks was provided by Miller~\cite{Miller2017}.

    One of the equations in system~\eqref{eq: system trivial alex} is \(\sigma_{n-1}(q_1,\dots,q_n) = (-1)^{\frac{n-1}{2}}\), which corresponds to an equation for the knot determinant (see Corollary~\ref{cor: det}). For \( n=3 \), this is the only condition, and it reproduces the well-known equation
    \[
    q_1q_2 + q_1q_3 + q_2q_3 = -1,
    \]
    which characterizes pretzel knots with three twist blocks having a trivial Alexander polynomial. In this case, there are infinitely many known solutions; for example, for any odd integer \( p \), the triple \( (-p, 2p-1, 2p+1) \) satisfies the equation. 
    Finding solutions for \(n > 3\) is significantly more difficult\footnote{Note that it is comparatively easy to find an infinite series of pretzel knots with trivial determinant for an arbitrary number of twists, but trivializing the entire polynomial is much more restrictive.}.
    We first observe that if a tuple of odd integers \( (q_1, \dots, q_n) \) is a solution of the system~\eqref{eq: system trivial alex}, then the extended tuple \( (\pm 1, \mp 1, q_1, \dots, q_n) \) yields a solution in \( n+2 \) variables. This is obvious from a geometric point of view --- it is nothing but an application of the second Reidemeister move.     
    In fact, if \((q_1,\dots,q_n)\) is a solution of the system~\eqref{eq: system trivial alex} and \(q_i = \pm 1\) for some \(i\), this forces \(q_j = \mp 1\) for some other \(j\). Indeed, let \(P(t) := \prod_{j=1}^{n} (t - q_j)\). From the equations, one checks that \(P(t) = E(t^2) + t (t^2 - 1)^n\) for some polynomial \(E(t)\), and hence \(P(1) = P(-1)\). Thus, if \(q_i = \pm1\) for some \(i\), then \(P(\pm1)=0\), which implies \(P(\mp1)=0\); consequently, there exists some \(j\) such that \(q_j = \mp1\). 
    Therefore, it is natural to investigate \emph{irreducible} solutions, which we define as the solutions of system~\eqref{eq: system trivial alex} in which \(|q_i|>1\) for all \(i\).

    We performed a computational search for such solutions. For \(n=5\), we searched through all tuples where the three smallest absolute values of the parameters do not exceed \(100\,001\), and we identified \numknownsol irreducible solutions (up to sign and permutation of the parameters)\footnote{The calculation was done with a {C++} program; the code of the program and the results of the calculations are available at~\cite{Bcode}.}. These solutions correspond to non-trivial knots due to the classification of Montesinos knots~\cite{Zieschang1984} (alternatively, the span of their Jones polynomial is non-zero --- see \cite[Theorem 9, Case 2]{DiazManchon2023}). From the work~\cite{Bryant2017} it also follows that all corresponding knots are not smoothly slice. Thus, we get a new series of topologically slice, but not smoothly slice, knots. The solutions are listed in Appendix~\ref{appendix}.

    For \(n=7\), we searched through all tuples where the four smallest absolute values of the parameters do not exceed \(5\,001\) and found no solutions. Based on these computations, we propose the following conjecture:

    \begin{conjecture}
       There are infinitely many irreducible solutions of the system~\eqref{eq: system trivial alex} for \(n=5\), while there are no irreducible solutions for \(n>5\).
    \end{conjecture}
	\bibliography{biblio}
	\bibliographystyle{alpha}

    \newpage
    \appendix 
    \section{Irreducible solutions of the system~\eqref{eq: system trivial alex} for \(n=5\)}\label{appendix}
    Since the equations in system~\eqref{eq: system trivial alex} are invariant under permutation of the variables and under multiplication of all variables by \(-1\), and since we are interested in irreducible solutions only, we present the solutions in the normalized form:
    \[
        1 < q_1 < |q_2| < |q_3| <|q_4| <|q_5|.    
    \]
    Note that we use strict inequalities here, since one of the equations in system~\eqref{eq: system trivial alex} is \(\sigma_4(q_1,\dots,q_5) = 1\), which forces \(q_1,\dots,q_5\) to be pairwise coprime. \medskip

    The list of  known irreducible solutions (\numknownsol in total):
    {\small
\[
\renewcommand{\arraystretch}{1} 
\begin{array}{rrrrr}
13, & -15, & -17, & 29, & 71 \\
89, & -109, & -111, & 271, & 307 \\
169, & -181, & -239, & 307, & 1\,871 \\
229, & -239, & -305, & 341, & 6\,119 \\
231, & -251, & -289, & 349, & 4\,001 \\
265, & -287, & -307, & 351, & 8\,399 \\
265, & -287, & -417, & 911, & 989 \\
415, & -417, & -1\,519, & 2\,393, & 4\,369 \\
573, & -575, & -2\,377, & 4\,159, & 5\,741 \\
951, & -1\,121, & -1\,189, & 1\,937, & 6\,049 \\
985, & -1\,189, & -1\,231, & 2\,379, & 4\,591 \\
1\,119, & -1\,219, & -1\,343, & 1\,585, & 24\,767 \\
1\,189, & -1\,231, & -2\,377, & 4\,591, & 5\,741 \\
1\,583, & -1\,807, & -1\,891, & 2\,485, & 20\,791 \\
2\,241, & -2\,243, & -11\,969, & 14\,279, & 76\,229 \\
2\,815, & -2\,899, & -4\,351, & 4\,929, & 60\,031 \\
3\,345, & -3\,347, & -23\,561, & 33\,461, & 80\,783 \\
4\,159, & -4\,929, & -5\,279, & 9\,569, & 21\,113 \\
4\,181, & -4\,673, & -5\,985, & 11\,969, & 17\,137 \\
4\,351, & -4\,929, & -5\,851, & 9\,569, & 25\,345 \\
4\,369, & -4\,645, & -6\,449, & 8\,371, & 45\,449 \\
5\,509, & -5\,851, & -8\,399, & 11\,311, & 49\,895 \\
8\,161, & -8\,373, & -17\,065, & 25\,439, & 61\,777 \\
8\,849, & -8\,969, & -23\,599, & 41\,079, & 60\,535 \\
9\,793, & -11\,329, & -12\,979, & 24\,309, & 45\,319 \\
10\,319, & -11\,439, & -15\,049, & 29\,413, & 43\,549 \\
11\,285, & -11\,593, & -24\,243, & 43\,471, & 62\,929 \\
11\,323, & -13\,111, & -15\,199, & 34\,271, & 40\,699 \\
12\,431, & -12\,655, & -25\,991, & 32\,339, & 163\,171 \\
13\,201, & -14\,169, & -20\,929, & 35\,099, & 70\,849 \\
20\,501, & -20\,747, & -38\,457, & 41\,999, & 619\,345 \\
23\,561, & -24\,551, & -26\,599, & 28\,271, & 1\,953\,199 \\
26\,351, & -30\,499, & -34\,039, & 55\,999, & 157\,249 \\
29\,303, & -32\,339, & -43\,471, & 91\,909, & 112\,111 \\
34\,453, & -40\,019, & -43\,791, & 69\,551, & 226\,199 \\
43\,777, & -44\,839, & -60\,535, & 64\,767, & 1\,857\,439 \\
76\,517, & -87\,891, & -97\,009, & 144\,535, & 588\,817 \\
76\,735, & -83\,711, & -96\,139, & 117\,503, & 1\,241\,953
\end{array}
\]
}
\end{document}